\documentclass[11pt]{amsart}
\numberwithin{equation}{section}
\newtheorem{thm}{Theorem}[section]
\newtheorem{lem}[thm]{Lemma}

\newtheorem{conj}[thm]{Conjecture}
\theoremstyle{definition}
\newtheorem{defn}[thm]{Definition}
\newtheorem{exmp}[thm]{Example}
\newtheorem{rem}[thm]{Remark}

\newcommand{\bbC}{\mathbb{C}}
\newcommand{\bbZ}{\mathbb{Z}}
\newcommand{\bP}{\mathbf{P}}
\newcommand{\bQ}{\mathbf{Q}}
\newcommand{\fg}{\mathfrak{g}}
\newcommand{\hg}{\hat{\mathfrak{g}}}
\newcommand{\hsl}{\hat{\mathfrak{sl}}}
\newcommand{\ol}[1]{\overline{#1}}

\title[On Frenkel-Mukhin algorithm]{On Frenkel-Mukhin algorithm \\
for $q$-character of quantum affine algebras}

\author[W.\ Nakai and T.\ Nakanishi]{Wakako Nakai and Tomoki Nakanishi
}

\address{
{\rm Wakako Nakai}\\
Kaiyo Academy\\
Kaiyo-cho, Gamagori, 443-8588, Japan
\\
{\rm Tomoki Nakanishi}\\
Graduate School of Mathematics, Nagoya University\\
Chikusa-ku, Nagoya, 464-8604, Japan}

\dedicatory{Dedicated to Professor Akihiro Tsuchiya}

\begin{document}

\begin{abstract}
The $q$-character is a strong tool
to study  finite-dimensional representations
of quantum affine algebras.
However, the explicit formula of the
$q$-character of a given representation
has not been known so far.
Frenkel and Mukhin proposed the iterative algorithm which generates
the $q$-character of a given irreducible representation
starting from its highest weight monomial.
The algorithm is known to work for various classes of
representations.
In this note, however, we give an example in which the algorithm
fails to generate the $q$-character.

\end{abstract}

\maketitle

\section{Background}

\subsection{Finite-dimensional representations of
quantum affine algebras}
Let $\fg$ be a simple Lie algebra over $\bbC$,
and let $U_q(\hg)$ be the untwisted
quantum affine algebra of $\fg$  by
Drinfeld and Jimbo \cite{D1, D2, J}.

The following
are the most basic facts on
 the finite-dimensional representations  of $U_q(\hg)$,
due to Chari-Pressley \cite{CP1, CP2}:

(i) The isomorphism classes
 of the irreducible  finite-dimensional representations
of $U_q(\hg)$
 are
parametrized by an $n$-tuple of polynomials
of constant term 1,
$\bP = (P_i(u))_{i\in I}$, where $I=\{1,\dots,n\}$
and $n=\mathrm{rank}\, \fg$.
The polynomials $\bP$ are often
 called the {\em Drinfeld polynomials}
because an analogous result for Yangian was obtained 
earlier by Drinfeld \cite{D2}.
\par
(ii) For given Drinfeld polynomials $\bP$,
let $V(\bP)$
denote the corresponding irreducible representation.
For a pair of Drinfeld polynomials $\bP = (P_i(u))_{i\in I}$ 
and $\bQ = (Q_i(u))_{i\in I}$, let $\bP\bQ:=(P_i(u)Q_i(u))_{i\in I}$.
Then, $V(\bP\bQ)$ is a subquotient of $V(\bP)\otimes V(\bQ)$.
\par
(iii)  A representation $V(\bP)$ is called the $i${\rm th}
{\em fundamental representation} and denoted by
 $V_{\omega_i}(a)$ if $P_i(u)= 1-au$ and 
$P_j(u)=1$ for any $j \neq i$.
Suppose that
Drinfeld polynomials $\bP$
are in the form
\begin{equation}
\label{eq:DrinfeldPolynomials}
P_i(u)=\prod_{k=1}^{n_i} (1-a^{(i)}_{k} u).
\end{equation}
Namely, $a^{(i)}_k$ are the inverses of the zeros of $P_i(u)$.
Then, as a consequence of (ii),
$V(\bP)$ is a subquotient of the tensor product
of fundamental representations
$
\bigotimes_{i\in I}
\bigotimes_{k=1}^{n_i}
V_{\omega_i}(a^{(i)}_{k})$.

For $\fg $ of type $A_1$, the structure of $V(\bP)$ for an arbitrary $\bP$
is known \cite{CP1}.
Also, when $\fg$ is simply-laced,
the relation between $V(\bP)$ and the so-called
standard representations
is described by an analogue of the Kazhdan-Lusztig polynomials
\cite{N1}.
So far, no more general results are known
for the  structure of $V(\bP)$.

\subsection{$q$-Character}

To study the structure of $V(\bP)$,
the {\it $q$-character} $\chi_q$
was introduced by Frenkel and Reshetikhin \cite{FR}.
It is an injective ring homomorphism from the Grothendieck
ring of the finite-dimensional representations of $U_q(\hg)$
to the Laurent polynomial ring of infinitely-many variables
$Y_{i,a}$, $i\in I, a\in \bbC^\times$,
\begin{equation}
\chi_q : \mathrm{Rep}\, U_q(\hg) \rightarrow
\bbZ[Y_{i,a}^{\pm1}]_{i\in I; a\in \bbC^\times}.
\end{equation}
The variables $Y_{i,a}$ are regarded as affinizations of 
the formal exponentials $\exp(\omega_i)$
of the fundamental weights $\omega_i$ of $U_q(\fg)$.
By replacing $Y_{i,a}$ with $\exp(\omega_i)$, 
$\chi_q(V(\bP))$ reduces to the underlying $U_q(\fg)$-character
of $V(\bP)$   with respect to the
standard embedding $U_q(\fg) \subset U_q(\hg)$.

There are several equivalent ways to define the $q$-character.

(i) {\it By universal $R$-matrix.}
This is the original definition of \cite{FR}. 
The idea originates from the {\it transfer matrix},
which plays the central role
in  the {\it quantum inverse
scattering method}, or the {\it  Bethe ansatz method}
for integrable spin chains such as the Heisenberg $XXX$ model \cite{TF}.
The $q$-character $\chi_q(V)$ of 
a representation $V$
is defined as a partial trace of the universal
$R$-matrix of $U_q(\hg)$ on $V$.

(ii) {\it By  weight decomposition.}
 It is shown also  in \cite{FR} that
$\chi_q(V)$ is regarded as the formal character of the weight decomposition
of $V$ with respect to
certain elements 
in the Cartan subalgebra in the `second realization'
of $U_q(\hg)$ \cite{D2}.
Hernandez extended this definition of $\chi_q$ to 
the affinizations of the full family of the quantum
Kac-Moody algebras \cite{H2,H4}.

(iii) {\it By quiver varieties.}
When $\fg$ is simply-laced, Nakajima \cite{N1,N2} geometrically
defined a
$t$-analogue of $q$-character  $\chi_{q,t}$
(the {\it $q,t$-character\/})
as the generating function of the Poincar\'e polynomials of
graded quiver varieties.
Then, $\chi_q$ is obtained by $\chi_q=\chi_{q,1}$.
The algorithm of calculating $\chi_{q,t}$ is given
based on the analogue of the Kazhdan-Lusztig  polynomials.

(iv)  {\it By axiom.}
In \cite{N2} the axiom which characterizes  $\chi_{q,t}$  in (iii)
is given. The axiom is further extended for non simply-laced cases
in \cite{H1}.
Then, $\chi_q$ is obtained by $\chi_q=\chi_{q,1}$.

Before the introduction of the $q$-character,
the spectrum of 
the transfer matrix defined by
the trace on a so-called
{\it Kirillov-Reshetikhin (KR)
representation}
 \cite{KR} of $U_q(\hg)$ was extensively studied 
by the  Bethe ansatz method
(\cite{R1,R2,R3,BR,KR,KNS,KNH,KOS,KS,TK}, etc.).
The fundamental representations $V_{\omega_i}(a)$, for example,
are special cases of the KR representations.
Because of Definition (i) above, these results, including
many conjectures, are naturally translated and restudied
in the context of the $q$-character \cite{FR,FM1,CM,KOSY,N3,H3,H6}.
As a result, the $q$-characters of the KR representations are,
not fully, but rather well understood now.

However, beyond the KR representations, not much is known for
the explicit formula of the $q$-character
 except for some partial results and conjectures
(e.g., \cite{H5,NN1,NN2,NN3}).

\subsection{Frenkel-Mukhin algorithm}

We say that
a monomial in $\bbZ[Y_{i,a}^{\pm1}]_{i\in I; a\in \bbC^\times}$
is {\it dominant} if it is a monomial of variables $Y_{i,a}$,
$i\in I, a\in \bbC^\times$, i.e., without
$Y^{-1}_{i,a}$.
Suppose that Drinfeld polynomials $\bP$
are in the form (\ref{eq:DrinfeldPolynomials}).
Then,  $\chi_q(V(\bP))$ contains a dominant monomial
\begin{equation}
\label{eq:hwm}
m_+ =\prod_{i \in I}
\prod_{k=1}^{n_i} Y_{i,a_k^{(i)}}
\end{equation}
called the {\it highest weight monomial\/} of $V(\bP)$  \cite{FR}.
Since  $\bP$ and $m_+$ are in one-to-one correspondence,
we parametrize the irreducible 
representations of $U_q(\hg)$ by their highest weight monomials
as  $V(m_+)$, instead of
$V(\bP)$, from now on.

Frenkel and Mukhin \cite{FM1} introduced
the iterative algorithm
 which generates a polynomial, say,  $\chi(m_+)
 \in \bbZ[Y_{i,a}^{\pm1}]_{i\in I; a\in \bbC^\times}$
{}from a  given dominant monomial $m_+$.
We call it the {\it FM algorithm} here.
{\it A priori}, it is not clear whether
the algorithm  does not fail (i.e., it is not halted halfway);
also it is not clear whether it  stops at finitely many steps.
It was conjectured that

\begin{conj}[\cite{FM1}, Conjecture 5.8]
\label{conj:FMconjecture}
For any dominant monomial $m_+$,
the algorithm never fails and stops after finitely many steps.
Moreover, the result $\chi(m_+)$ equals to  $\chi_q(V(m_+))$.
\end{conj}

The algorithm is fairly practical so that,
assuming the conjecture,
one can explicitly calculate
the $q$-characters of representations,
by hand, or by computer,
when the dimensions are small.

Conjecture \ref{conj:FMconjecture}
 is partially proved by \cite{FM1} as we shall  explain now.
 We say a representation $V(m_+)$ is {\it special\/} if
its highest weight monomial
$m_+$ is the unique dominant monomial occurring in $\chi_q(V(m_+))$.
For example,
the fundamental representations are special \cite{FM1}.
More generally, the KR representations are special \cite{N3,H3,H6}.
(See \cite{H5} for further examples of special
representations.)
\begin{thm}[\cite{FM1}, Theorem 5.9]
If $V(m_+)$ is special, then Conjecture \ref{conj:FMconjecture}
is true.
\end{thm}
In particular,
the FM algorithm is applicable 
to the fundamental representations and the KR representations,
and provides the aforementioned results for their
$q$-characters.
We note that
there are also many  {\it nonspecial} representations
for which Conjecture \ref{conj:FMconjecture} is true;
{\it e.g.}, $\fg$ of type $A_2$ with $m_+ = Y^2_{1,1}Y_{1,q^2}$,
where $V(m_+)\simeq V(Y_{1,1})\otimes V(Y_{1,1}Y_{1,q^2})$.

The purpose of this note is to give a counterexample
of Conjecture \ref{conj:FMconjecture}.
More precisely, it is an example where the algorithm
{\it fails\/} in the sense of \cite{FM1} (see
Definition \ref{defn:FMalgorithm}).

In Section 2 the FM algorithm is recalled.
In Section 3, as a warmup,
we give two examples in which
the algorithm works well.
Then, a counterexample is given in Section 4.
Taking this opportunity, we also demonstrate
 the synthesis of the FM algorithm and
Young tableaux in
\cite{BR,KOS,KS,NT,NN1,NN2,NN3}
by these examples.

We thank D.\ Hernandez for his helpful comments.

\section{FM algorithm}

Here we recall the FM algorithm.
The presentation here is minimal to describe the
counterexample in Section 4.
We faithfully follow  \cite[Section 5.5]{FM1},
so that the reader is asked
to consult  it for more details.

\subsection{Preliminary: $q$-character of $U_q(\hsl_2)$}
The FM algorithm is based on the explicit formula
of the $q$-characters of
the irreducible representations of $U_q(\hsl_2)$ \cite{CP1,FR}.

\begin{exmp}
Let $W_r(a)$ be the irreducible representation $U_q(\hsl_2)$
with highest weight monomial
\begin{align}
\label{eq:hwmW}
m_+= \prod_{k=1}^r Y_{aq^{r-2k+1}},
\end{align}
where we set $Y_{a}:=Y_{1,a}$.
Then, its $q$-character is given by
\begin{align}
\label{eq:qchW}
\chi_q(W_r(a)) = 
m_+
\sum_{i=0}^r \prod_{j=1}^i
A^{-1}_{aq^{r-2j+2}},
\quad
A_{a}:= Y_{aq^{-1}} Y_{aq}.
\end{align}
\end{exmp}

Generally, the $q$-character of any irreducible
 representation of $U_q(\hsl_2)$ is given by  a product
of (\ref{eq:qchW}) as follows \cite{CP1}:
Let 
$\Sigma_{a,r}$
be the set of the indices of the
variables $Y_{b}$ in (\ref{eq:hwmW}),
i.e., $\Sigma_{a,r}=\{ aq^{r-2k+1}\}_{k=1,\dots,r}$.
We call it a {\it $q$-string}.
We say that two $q$-strings
 $\Sigma_{a,r}$ and $\Sigma_{a',r'}$ are
in {\it general position\/} if either 
(i) the union $\Sigma_{a,r} \cup\Sigma_{a',r'}$ is not a
$q$-string, or (ii)
$\Sigma_{a,r} \subset \Sigma_{a',r'}$ or
$\Sigma_{a',r'} \subset\Sigma_{a,r}$.
Then,

\begin{exmp}
Let
$m_+ \in  \bbZ[Y^{\pm1}_{a}]_{a\in \bbC^\times}$
be a given dominant monomial.
Then, 
one can uniquely (up to permutations) factorize
$m_+$ as
\begin{align}
m_+=\prod_{i=1}^k
\left(\prod_{b\in \Sigma_{a_i,r_i}} Y_{b}\right),
\end{align}
where $\Sigma_{a_1,r_1}$,\dots, $\Sigma_{a_k,r_k}$
are $q$-strings which are pairwise in general
position.
The $q$-character of $V(m_+)$ is given by
\begin{align}
\label{eq:chgeneral}
\chi_q(V(m_+)) = \prod_{i=1}^k \chi_q(W_{r_i}(a_i)).
\end{align}

\end{exmp}

\subsection{Algorithm}
Let us start from some key definitions.
\begin{defn}
(i) We say that
a monomial $m\in \bbZ[Y_{i,a}^{\pm1}]_{i\in I; a\in \bbC^\times}$
is  {\it $i$-dominant} if it does not contain
variables $Y^{-1}_{i,a}$,
$a\in \bbC^\times$.
\par
(ii)
For a polynomial $\chi \in \bbZ[Y_{i,a}^{\pm1}]_{i\in I; a\in \bbC^\times}$
and a monomial $m$ occurring in $\chi$ with coefficient $s$,
a {\it coloring\/} of $m$ is a set of integers
$\{s_i\}_{i\in I}$ such that
$0\leq s_i\leq s$.
We say that
a polynomial $\chi$ is {\it colored\/} if all monomials occurring in $\chi$
have colorings.
\par
(iii)
Let $\bbZ[Y_{i,a}^{\pm1}]_{i\in I; a\in \bbC^\times}$
be a colored polynomial,
and let $m$ be a monomial occurring in $\chi$ with coefficient
$s\in \bbZ_{\geq 0}$ and coloring $\{s_i\}_{i\in I}$.
We say that $m$ is  {\it admissible\/} if,
for any $i\in I$ such that $s_i< s$, $m$ is $i$-dominant.
\end{defn}
Let
\begin{align}
\label{eq:Amonomial}
\begin{split}
A_{i,a} &= Y_{i,aq_i^{-1}}Y_{i,aq_i}
\prod_{j:C_{ji}=-1} Y^{-1}_{j,a}\\
&\qquad
\times\prod_{j:C_{ji}=-2} Y^{-1}_{j,aq^{-1}}Y^{-1}_{j,aq}
\prod_{j:C_{ji}=-3} Y^{-1}_{j,aq^{-2}}Y^{-1}_{j,a}Y^{-1}_{j,aq^2},
\end{split}
\end{align}
where $C_{ij}=2(\alpha_i,\alpha_j)/(\alpha_i,\alpha_i)$
 is the Cartan matrix of $\fg$.
The monomials $A_{i,a}$ are 
regarded as affinizations of 
the formal exponentials $\exp(\alpha_i)$
of the simple roots $\alpha_i$ of $U_q(\fg)$.

The FM algorithm is an iterative algorithm,
and its main routine utilizes
the  following procedure called the {\it $i$-expansion}:
\begin{defn}
\label{defn:i-expansion}
Let $i\in I$,
$\chi$ be a colored polynomial,
and $m$ be an admissible monomial occurring in $\chi$
with coefficient $s$ and coloring $\{s_j\}_{j\in I}$.
Then, a new colored polynomial $i_m(\chi)$,
called the {\it $i$-expansion 
of $\chi$ with respect to $m$},
is defined as follows:
\par
(i) If $s_i=s$, then $i_m(\chi)=\chi$.
\par
(ii) If $s_i< s$, we define $i_m(\chi)$
in the following two steps.

First, we obtain a
colored polynomial $\mu$ which depends on 
$m$ and $i$ (but not on $\chi$) as follows:
Let $\overline{m}$ be the $i$th projection of $m$,
i.e., $\overline{Y}^{\pm 1}_{i,a} = Y^{\pm 1}_{a}$
and  $\overline{Y}^{\pm 1}_{j,a} = 1$ for any $j\neq i$.
Let $\chi_{q_i}(V(\overline{m}))
=\overline{m}(1+\sum_p \overline{M}_p)$
be the $q$-character of the irreducible
representation $V(\overline{m})$
of  $U_{q_i}(\hsl_2)$ with highest weight monomial $\overline{m}$,
where $\overline{M}_p$ is a product of $\overline{A}_{i,a}^{-1}$
(see \ (\ref{eq:qchW}), (\ref{eq:chgeneral}),
and (\ref{eq:Amonomial})).
Then, 
\begin{equation}
\label{eq:sl2ch}
\mu = m (1+ \sum_p M_p),
\end{equation}
where $M_p$ is obtained from $\overline{M}_p$ by replacing
all $\overline{A}_{i,a}^{-1}$ by $A_{i,a}^{-1}$
in (\ref{eq:Amonomial}).

Next, we obtain $i_m(\chi)$  by 
adding the monomials occurring in $\mu$ to $\chi$ as follows:
Suppose that a monomial $n$  occurs in $\mu$ with coefficient $t$.
If $n$ does not occur in $\chi$, we add $n$ to $\chi$
with coefficient $t(s-s_i)$ and set its coloring
$\{ s'_j\}_{j\in I}$ as
 $s'_j=0$ for any $j\neq i$ and $s'_i=t(s-s_i)$.
If $n$ occurs in $\chi$ with coefficient $r$ and coloring 
$\{r_j\}_{j\in I}$,
we set the coefficient $s'$ and the coloring
$\{ s'_j\}_{j\in I}$ of $n$ in $i_m(\chi)$
as $s'=\max\{r,r_i+t(s-s_i)\}$, 
$s'_j=r_j$ for any $j\neq i$, and
$s'_i = r_i + t(s-s_i)$.
The coefficients and the colorings of other monomials
occurring in $\chi$ are unchanged in $i_m(\chi)$.
\end{defn} 

Note that the $i$-expansion is defined only if
 $m$ is admissible.

\begin{rem}
\label{rem:saturate}
In (\ref{eq:sl2ch}),
the coefficient of $m$ in $\mu$ is always 1.
Therefore, both
the coefficient and the $i$th coloring of $m$ in $i_m(\chi)$
are $s$ in Definition \ref{defn:i-expansion}.
In other words, the $i$-expansion of $\chi$ with respect to $m$
is designed to saturate the $i$th coloring of $m$ to its coefficient.
\end{rem}

\begin{defn}
(i) 
The {\it $U_q(\fg)$-weight} of a monomial
\begin{equation*}
\prod_{i\in I}
\left(
\prod_{r=1}^{k_i} Y_{i,a_{ir}}
\prod_{s=1}^{l_i} Y^{-1}_{i,b_{is}}
\right)
\end{equation*}
is defined by 
$
\sum_{i\in I}
(k_i-l_i)\omega_i.
$
\par
(ii)  We equip the $U_q(\fg)$-weight lattice $P:=\bigoplus_{i\in I}
\bbZ \omega_i$  with a  partial order such that
   $\lambda \geq \lambda'$ if $\lambda-\lambda' = 
\sum_i a_i \alpha_i$, $a_i \in \bbZ_{\geq 0}$,
and call it the {\it natural partial order} in $P$.
\end{defn}

Now let us define the FM algorithm.
It is an algorithm generating
a colored polynomial $\chi(m_+)
\in \bbZ[Y_{i,a}^{\pm1}]_{i\in I; a\in \bbC^\times}$ from
a given dominant monomial $m_+$.

\begin{defn}[The FM algorithm]
\label{defn:FMalgorithm}
Let $m_+$ be a given dominant monomial
in $ \bbZ[Y_{i,a}^{\pm1}]_{i\in I; a\in \bbC^\times}$,
and $\lambda_+$
 be the $U_q(\fg)$-weight of $m_+$.
Choose any total order in 
 the set $P_{\leq \lambda_+}:=\{\mu \in P \mid
\mu \leq \lambda_+ \}$ such that
it is compatible
with the natural partial order in $P$; then
enumerate the elements
in $P_{\leq \lambda_+}$ as
$\lambda_1=\lambda_+ > \lambda_2 > \lambda_3 > \dots$.

{\it Step 1.} We set the colored polynomial $\chi$
by  $\chi=m_+$ with the $i$th coloring of $m_+$
being $0$ for any $i\in I$.

{\it Step 2.} Repeat the following steps
(i)--(iii) for $\lambda=\lambda_1,
\lambda_2,\lambda_3,\dots$.
\begin{itemize}
\item[(i)]
Let $\chi$ be the colored polynomial obtained in
the previous step.
Let $m_1, \dots, m_t$ be
all the monomials occurring in $\chi$
whose $U_q(\fg)$-weights are $\lambda$.
If there is at least one {\it non-admissible} monomial among them,
then the algorithm halted halfway.
We say that {\it the algorithm fails at $m_i$} if $m_i$ is one of
such non-admissible monomials.

\item[(ii)] Repeat the following for
all $i\in I$ and all $k\in {1,\dots,t}$:
Replace $\chi$ with the $i$-expansion $i_{m_k}(\chi)$ of $\chi$
with respect to $m_k$.
\item[(iii)] If there is no monomial occurring in $\chi$
whose $U_q(\fg)$-weight is less than $\lambda$ in the total
order of $P_{\leq \lambda_+}$,
then set $\chi(m_+)=\chi$ and
the algorithm {\it stops} (i.e., completes).
\end{itemize}
\end{defn}

It follows from  Remark \ref{rem:saturate} that,
if the algorithm successfully stops,
the $i$th coloring of any monomial $m$
occurring in $\chi(m_+)$
equals to the coefficient of $m$
for any $i\in I$.
Thus, once $\chi(m_+)$ is obtained,
one can safely forget the coloring.

\section{Examples}

Let us see how the FM algorithm works in good situations.
This  is  a warmup to understand the `bad situation'
in the next section.

\subsection{Example 1}
Let us consider the case where $\fg$ is of type $A_2$
and the representation $V(m_+)$
has the highest weight monomial
\begin{align}
\label{eq:A2highest}
m_+ = Y_{1,q^2}Y_{2,q^{-1}}.
\end{align}
The $U_q(\fg)$-weight of $m_+$ is $\lambda_+ = \omega_1 + \omega_2$.
It is well known that $V(m_+)$ is an
{\it evaluation representation} of the
adjoint representation $V_{\omega_1 + \omega_2}$
of $U_q(\fg)$.
As a $U_q(\fg)$-representation,
it is  isomorphic to $V_{\omega_1 + \omega_2}$.
It is also known that $V(m_+)$ is special \cite{H5} so
that the FM algorithm is applicable.

We use the following data:
$q_1=q_2=q$,
$\alpha_1  = 2\omega_1 - \omega_2$,
$\alpha_2  = -\omega_1 +2 \omega_2$,
and
\begin{align}
\label{eq:A1}
A^{-1}_{1,a}= Y^{-1}_{1,aq^{-1}}Y^{-1}_{1,aq}Y_{2,a},
\quad
A^{-1}_{2,a}= Y^{-1}_{2,aq^{-1}}Y^{-1}_{2,aq}Y_{1,a}.
\end{align}

Now let us execute the FM algorithm step by step.
We choose a total order in $P_{\leq \lambda_+}$ as
\begin{equation}
\begin{split}
\lambda_1&=\lambda_+,
\quad
\lambda_2=\lambda_+-\alpha_1,
\quad
\lambda_3=\lambda_+-\alpha_2,
\lambda_4=\lambda_+-2\alpha_1,\\
\lambda_5&=\lambda_+-\alpha_1-\alpha_2,
\quad
\lambda_6=\lambda_+-2\alpha_2,
\quad
\lambda_7=\lambda_+-2\alpha_1-\alpha_2,
\\
\lambda_8&=\lambda_+-\alpha_1 - 2\alpha_2,
\quad
\lambda_9=\lambda_+-2\alpha_1 - 2\alpha_2,
\quad\dots,
\end{split}
\end{equation}
where the rest of the order is irrelevant.

{\it Step 1.} Set $\chi = m_+ =  Y_{1,q^2}Y_{2,q^{-1}}$
 with the coloring of $m_+$ being $(0,0)$.

{\it Step 2.} (1) $\lambda = \lambda_1 = \omega_1 + \omega_2$. 

The 1-expansion of $\chi$ with respect to
 $ m_+=Y_{1,q^2}Y_{2,q^{-1}}$ is done as follows:
Since $\overline{ Y_{1,q^2}Y_{2,q^{-1}}} = Y_{q^2}$,
we have
\begin{align*}
\chi_q(V)  &= Y_{q^2}(1 + \overline{A}^{-1}_{1,q^3}),\\
\mu  &= Y_{1,q^2}Y_{2,q^{-1}}(1 + A^{-1}_{1,q^3})
 = Y_{1,q^2}Y_{2,q^{-1}} + Y^{-1}_{1,q^4}Y_{2,q^{-1}}Y_{2,q^{3}},\\
1_{m_+}(\chi)
 &= Y_{1,q^2}Y_{2,q^{-1}} + Y^{-1}_{1,q^4}Y_{2,q^{-1}}Y_{2,q^{3}},\\
  &\qquad (1,0)\hskip35pt (1,0) 
\end{align*}
where $(1,0)$ represents the coloring.
Then, $\chi$ is replaced with $1_{m_+}(\chi)$.

Similarly, the 2-expansion of $\chi$ 
with respect to $m_+$ is calculated as
\begin{align*}
\mu  &= Y_{1,q^2}Y_{2,q^{-1}}(1 + A^{-1}_{2,1})
 = Y_{1,q^2}Y_{2,q^{-1}} +
Y_{1,1}Y_{1,q^{2}}Y^{-1}_{2,q},\\
\chi &= Y_{1,q^2}Y_{2,q^{-1}} + Y^{-1}_{1,q^4}Y_{2,q^{-1}}Y_{2,q^{3}}
+ Y_{1,1}Y_{1,q^{2}}Y^{-1}_{2,q}.\\
  &\qquad (1,1)\hskip35pt (1,0) \hskip55pt (0,1)
\end{align*}

(2) $\lambda = \lambda_2 = -\omega_1 + 2\omega_2$. 
{}From now on, we only write down the nontrivial $i$-expansions,
i.e., the cases where $s_i <s$.

The 2-expansion w.r.t.\ $Y^{-1}_{1,q^4} Y_{2,q^{-1}} Y_{2,q^3}$:
\begin{align*}
\mu & =Y^{-1}_{1,q^4} Y_{2,q^{-1}} Y_{2,q^3}
(1 + A^{-1}_{2,1})(1 + A^{-1}_{2,q^4})\\
& = Y^{-1}_{1,q^4} Y_{2,q^{-1}} Y_{2,q^3}
+Y_{1,1}Y^{-1}_{1,q^4}Y^{-1}_{2,q}Y_{2,q^3}
+ Y_{2,q^{-1}}Y^{-1}_{2,q^5} + Y_{1,1}Y^{-1}_{2,q}Y^{-1}_{2,q^5},\\
\chi &= Y_{1,q^2}Y_{2,q^{-1}} + Y^{-1}_{1,q^4}Y_{2,q^{-1}}Y_{2,q^{3}}
+ Y_{1,1}Y_{1,q^{2}}Y^{-1}_{2,q}
+Y_{1,1}Y^{-1}_{1,q^4}Y^{-1}_{2,q}Y_{2,q^3}
\\
  &\qquad (1,1)\hskip35pt (1,1) \hskip55pt (0,1) \hskip55pt (0,1)\notag\\
&\quad + Y_{2,q^{-1}}Y^{-1}_{2,q^5} + Y_{1,1}Y^{-1}_{2,q}Y^{-1}_{2,q^5}.
\notag\\
  &\quad\qquad (0,1)\hskip35pt (0,1) \notag
\end{align*}

(3) $\lambda = \lambda_3 = 2\omega_1  - \omega_2$. 

The 1-expansion w.r.t.\ $Y_{1,1} Y_{1,q^2} Y^{-1}_{2,q}$:
\begin{align*}
\mu & =Y_{1,1} Y_{1,q^2} Y^{-1}_{2,q}
(1 + A^{-1}_{1,q^3} + A^{-1}_{1,q}A^{-1}_{1,q^3})\\
& = Y_{1,1}Y_{1,q^{2}}Y^{-1}_{2,q}
+ Y_{1,1}Y^{-1}_{1,q^4}Y^{-1}_{2,q}Y_{2,q^3}
+ Y^{-1}_{1,q^2} Y^{-1}_{1,q^4}Y_{2,q^3},\\
\chi &= Y_{1,q^2}Y_{2,q^{-1}} + Y^{-1}_{1,q^4}Y_{2,q^{-1}}Y_{2,q^{3}}
+ Y_{1,1}Y_{1,q^{2}}Y^{-1}_{2,q}
+Y_{1,1}Y^{-1}_{1,q^4}Y^{-1}_{2,q}Y_{2,q^3}
\\
  &\qquad (1,1)\hskip35pt (1,1) \hskip55pt (1,1) \hskip55pt (1,1)\notag\\
&\quad + Y_{2,q^{-1}}Y^{-1}_{2,q^5} + Y_{1,1}Y^{-1}_{2,q}Y^{-1}_{2,q^5}
+ Y^{-1}_{1,q^2} Y^{-1}_{1,q^4}Y_{2,q^3}.
\notag\\
  &\quad\qquad (0,1)\hskip35pt (0,1)  \hskip50pt (1,0)\notag
\end{align*}

(4) $\lambda = \lambda_4 = -3\omega_1  +3 \omega_2$. 
No nontrivial $i$-expansions.

(5) $\lambda = \lambda_5 = 0$.

The 1-expansion w.r.t.\ $Y_{2,q^{-1}} Y^{-1}_{2,q^5}$:
\begin{align*}
\mu &= Y_{2,q^{-1}} Y^{-1}_{2,q^5},\\
\chi &= Y_{1,q^2}Y_{2,q^{-1}} + Y^{-1}_{1,q^4}Y_{2,q^{-1}}Y_{2,q^{3}}
+ Y_{1,1}Y_{1,q^{2}}Y^{-1}_{2,q}
+Y_{1,1}Y^{-1}_{1,q^4}Y^{-1}_{2,q}Y_{2,q^3}
\\
  &\qquad (1,1)\hskip35pt (1,1) \hskip55pt (1,1) \hskip55pt (1,1)\notag\\
&\quad + Y_{2,q^{-1}}Y^{-1}_{2,q^5} + Y_{1,1}Y^{-1}_{2,q}Y^{-1}_{2,q^5}
+ Y^{-1}_{1,q^2} Y^{-1}_{1,q^4}Y_{2,q^3}.
\notag\\
  &\quad\qquad (1,1)\hskip35pt (0,1)  \hskip50pt (1,0)\notag
\end{align*}

(6) $\lambda = \lambda_6 = 3\omega_1  -3 \omega_2$. 
No nontrivial $i$-expansions.

(7) $\lambda = \lambda_7 = -2\omega_1  + \omega_2$. 

The 2-expansion w.r.t.\ $ Y^{-1}_{1,q^2} Y^{-1}_{1,q^4}Y_{2,q^3}$:
\begin{align*}
\mu & =
 Y^{-1}_{1,q^2} Y^{-1}_{1,q^4}Y_{2,q^3}
(1 + A^{-1}_{2,q^4})
 =  Y^{-1}_{1,q^2} Y^{-1}_{1,q^4}Y_{2,q^3}
+ Y^{-1}_{1,q^2} Y^{-1}_{2,q^5},\\
\chi &= Y_{1,q^2}Y_{2,q^{-1}} + Y^{-1}_{1,q^4}Y_{2,q^{-1}}Y_{2,q^{3}}
+ Y_{1,1}Y_{1,q^{2}}Y^{-1}_{2,q}
+Y_{1,1}Y^{-1}_{1,q^4}Y^{-1}_{2,q}Y_{2,q^3}
\\
  &\qquad (1,1)\hskip35pt (1,1) \hskip55pt (1,1) \hskip55pt (1,1)\notag\\
&\quad + Y_{2,q^{-1}}Y^{-1}_{2,q^5} + Y_{1,1}Y^{-1}_{2,q}Y^{-1}_{2,q^5}
+ Y^{-1}_{1,q^2} Y^{-1}_{1,q^4}Y_{2,q^3}
 + Y^{-1}_{1,q^2} Y^{-1}_{2,q^5}.
\notag\\
  &\quad\qquad (1,1)\hskip35pt (0,1)  \hskip50pt (1,1)
\hskip55pt (0,1)\notag
\end{align*}

(8) $\lambda = \lambda_8 = \omega_1  -2 \omega_2$. 

The 1-expansion w.r.t.\ $Y_{1,1}Y^{-1}_{2,q}Y^{-1}_{2,q^5}$:
\begin{align}
\label{eq:chi8}
\begin{split}
\mu & =
Y_{1,1}Y^{-1}_{2,q}Y^{-1}_{2,q^5}
(1 + A^{-1}_{1,q})
 =  Y_{1,1}Y^{-1}_{2,q}Y^{-1}_{2,q^5}
+ Y^{-1}_{1,q^2} Y^{-1}_{2,q^5},\\
\chi &= Y_{1,q^2}Y_{2,q^{-1}} + Y^{-1}_{1,q^4}Y_{2,q^{-1}}Y_{2,q^{3}}
+ Y_{1,1}Y_{1,q^{2}}Y^{-1}_{2,q}
+Y_{1,1}Y^{-1}_{1,q^4}Y^{-1}_{2,q}Y_{2,q^3}
\\
  &\qquad (1,1)\hskip35pt (1,1) \hskip55pt (1,1) \hskip55pt (1,1)\\
&\quad + Y_{2,q^{-1}}Y^{-1}_{2,q^5} + Y_{1,1}Y^{-1}_{2,q}Y^{-1}_{2,q^5}
+ Y^{-1}_{1,q^2} Y^{-1}_{1,q^4}Y_{2,q^3}
 + Y^{-1}_{1,q^2} Y^{-1}_{2,q^5}.
\\
  &\quad\qquad (1,1)\hskip35pt (1,1)  \hskip50pt (1,1)
\hskip55pt (1,1)
\end{split}
\end{align}

(9) $\lambda = \lambda_9 = -\omega_1  - \omega_2$. 
There is no nontrivial $i$-expansions;
furthermore, there is no monomial occurring in $\chi$ in (\ref{eq:chi8})
whose $U_q(\fg)$-weight is less than $\lambda_9$
in the total order.
Therefore, we set $\chi(m_+)$ to be $\chi$
in (\ref{eq:chi8}), and the algorithm stops.

Thus, we obtain the $q$-character of $V(m_+)$
as $\chi$ in (\ref{eq:chi8}) by forgetting coloring.

\begin{figure}
\setlength{\unitlength}{0.2mm}
\begin{picture}(250,450)
\put(150,425){\line(0,1){25}}
\multiput(100,400)(25,0){2}{\line(0,1){50}}
\multiput(100,425)(0,25){2}{\line(1,0){50}}
\put(100,400){\line(1,0){25}}
\put(108,432){$1$}
\put(133,432){$1$}
\put(108,407){$2$}
\put(250,325){\line(0,1){25}}
\multiput(200,300)(25,0){2}{\line(0,1){50}}
\multiput(200,325)(0,25){2}{\line(1,0){50}}
\put(200,300){\line(1,0){25}}
\put(208,332){$1$}
\put(233,332){$1$}
\put(208,307){$3$}
\put(50,325){\line(0,1){25}}
\multiput(0,300)(25,0){2}{\line(0,1){50}}
\multiput(0,325)(0,25){2}{\line(1,0){50}}
\put(0,300){\line(1,0){25}}
\put(8,332){$1$}
\put(33,332){$2$}
\put(8,307){$2$}
\put(122,225){\line(0,1){25}}
\multiput(72,200)(25,0){2}{\line(0,1){50}}
\multiput(72,225)(0,25){2}{\line(1,0){50}}
\put(72,200){\line(1,0){25}}
\put(80,232){$1$}
\put(105,232){$3$}
\put(80,207){$2$}
\put(180,225){\line(0,1){25}}
\multiput(130,200)(25,0){2}{\line(0,1){50}}
\multiput(130,225)(0,25){2}{\line(1,0){50}}
\put(130,200){\line(1,0){25}}
\put(138,232){$1$}
\put(163,232){$2$}
\put(138,207){$3$}
\put(50,125){\line(0,1){25}}
\multiput(,100)(25,0){2}{\line(0,1){50}}
\multiput(0,125)(0,25){2}{\line(1,0){50}}
\put(0,100){\line(1,0){25}}
\put(8,132){$2$}
\put(33,132){$2$}
\put(8,107){$3$}
\put(250,125){\line(0,1){25}}
\multiput(200,100)(25,0){2}{\line(0,1){50}}
\multiput(200,125)(0,25){2}{\line(1,0){50}}
\put(200,100){\line(1,0){25}}
\put(208,132){$1$}
\put(233,132){$3$}
\put(208,107){$3$}
\put(150,25){\line(0,1){25}}
\multiput(100,0)(25,0){2}{\line(0,1){50}}
\multiput(100,25)(0,25){2}{\line(1,0){50}}
\put(100,0){\line(1,0){25}}
\put(108,32){$2$}
\put(133,32){$3$}
\put(108,7){$3$}
\put(60,90){\vector(1,-1){30}}
\put(190,90){\vector(-1,-1){30}}
\put(175,190){\vector(1,-2){15}}
\put(115,190){\vector(2,-1){60}}
\put(135,190){\vector(-2,-1){60}}
\put(55,290){\vector(1,-2){15}}
\put(75,290){\vector(2,-1){60}}
\put(195,290){\vector(-1,-2){15}}
\put(160,390){\vector(1,-1){30}}
\put(90,390){\vector(-1,-1){30}}
\put(35,65){$2,q^4$}
\put(185,65){$1,q$}
\put(200,175){$2,q^4$}
\put(130,150){$2,1$}
\put(45,175){$1,q$}
\put(20,265){$2,q^4$}
\put(120,275){$2,1$}
\put(200,265){$1,q^3$}
\put(195,375){$2,1$}
\put(30,375){$1,q^3$}
\put(140,408){$\underline{_{1,2}}$}
\put(40,308){$\underline{_{2}}$}
\put(240,308){$\underline{_{1}}$}
\put(112,208){$\underline{_{1}}$}
\put(40,108){$\underline{_{2}}$}
\put(240,108){$\underline{_{1}}$}
\end{picture}

\caption{ The flow of the FM algorithm by Young
tableaux for Example 1.
The symbol $i,q^k$ at an arrow represents
the action of $A^{-1}_{i,q^k}$.
The suffix $\underline{i}$ at a tableau indicates that
$s_i < s$ when $\chi$ is to be $i$-expanded
with respect to the corresponding monomial.}
\label{fig:A2}
\end{figure}
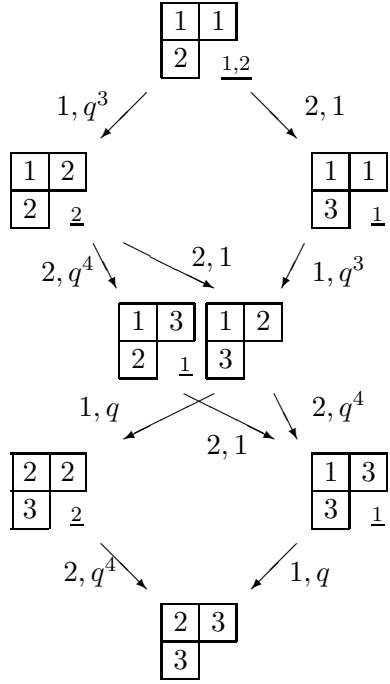

Next, let us introduce a diagrammatic notation
of monomials
by Young tableaux, following \cite{BR,KOS,KS,NT,FR,FM2,NN1}.

To each letter $a=1,2,3$ within a box
of a Young diagram $Y$,
 we assign a monomial as (cf.\ \cite[Section 5.4.1]{FR}) 
\begin{align}
\label{eq:box1}
\begin{split}
&
\raisebox{-4pt}
 {
{\setlength{\unitlength}{0.2mm}
\begin{picture}(25,25)
\multiput(0,0)(25,0){2}{\line(0,1){25}}
\multiput(0,0)(0,25){2}{\line(1,0){25}}
\put(8,7){$1$}
\put(28,-3){$_{ij}$}
\end{picture}
}
}
=
Y_{1,q^{-2i+2j}},
\
\raisebox{-4pt}
 {
{\setlength{\unitlength}{0.2mm}
\begin{picture}(25,25)
\multiput(0,0)(25,0){2}{\line(0,1){25}}
\multiput(0,0)(0,25){2}{\line(1,0){25}}
\put(8,7){$2$}
\put(28,-3){$_{ij}$}
\end{picture}
}
}
=
Y^{-1}_{1,q^{-2i+2j+2}}Y_{2,q^{-2i+2j+1}}
,
\\
&
\raisebox{-4pt}
 {
{\setlength{\unitlength}{0.2mm}
\begin{picture}(25,25)
\multiput(0,0)(25,0){2}{\line(0,1){25}}
\multiput(0,0)(0,25){2}{\line(1,0){25}}
\put(8,7){$3$}
\put(28,-3){$_{ij}$}
\end{picture}
}
}
=
Y^{-1}_{2,q^{-2i+2j+3}},
\end{split}
\end{align}
where the subscription `$ij$' indicates that
the box is located in  the $i$th row and $j$th column of  $Y$.
To each tableau $T$ on $Y$, we assign a monomial $m(T)$
by multiplying the monomials assigned to all the boxes
in $T$.
For example, 
the first and the second monomials in $\chi$ in 
(\ref{eq:chi8}) are represented by tableaux as
\begin{align}
m\left(
\raisebox{-13pt}
 {
{\setlength{\unitlength}{0.2mm}
\begin{picture}(50,50)
\put(50,25){\line(0,1){25}}
\multiput(0,0)(25,0){2}{\line(0,1){50}}
\multiput(0,25)(0,25){2}{\line(1,0){50}}
\put(0,0){\line(1,0){25}}
\put(8,32){$1$}
\put(33,32){$1$}
\put(8,7){$2$}
\end{picture}
}
}
\right)
&=
Y_{1,1} Y_{1,q^2} (Y^{-1}_{1,1}Y_{2,q^{-1}})
= Y_{1,q^2}Y_{2,q^{-1}},\\
m\left(
\raisebox{-13pt}
 {
{\setlength{\unitlength}{0.2mm}
\begin{picture}(50,50)
\put(50,25){\line(0,1){25}}
\multiput(0,0)(25,0){2}{\line(0,1){50}}
\multiput(0,25)(0,25){2}{\line(1,0){50}}
\put(0,0){\line(1,0){25}}
\put(8,32){$1$}
\put(33,32){$2$}
\put(8,7){$2$}
\end{picture}
}
}
\right)
&=
Y_{1,1} (Y^{-1}_{1,q^4}Y_{2,q^{3}})(Y^{-1}_{1,1}Y_{2,q^{-1}})
\\&
= Y^{-1}_{1,q^4}Y_{2,q^{-1}}Y_{2,q^{3}}.\notag
\end{align}
The definition  (\ref{eq:box1}) is designed so that
the following equalities hold \cite[Section 5.4.1]{FR}:
\begin{align}
\label{eq:aaction1}
A^{-1}_{1,q^{-2i+2j+1}}
{\raisebox{-4pt}
 {
{\setlength{\unitlength}{0.2mm}
\begin{picture}(25,25)
\multiput(0,0)(25,0){2}{\line(0,1){25}}
\multiput(0,0)(0,25){2}{\line(1,0){25}}
\put(8,7){$1$}
\put(28,-3){$_{ij}$}
\end{picture}
}
}
}
=
{\raisebox{-4pt}
 {
{\setlength{\unitlength}{0.2mm}
\begin{picture}(25,25)
\multiput(0,0)(25,0){2}{\line(0,1){25}}
\multiput(0,0)(0,25){2}{\line(1,0){25}}
\put(8,7){$2$}
\put(28,-3){$_{ij}$}
\end{picture}
}
}
},
\quad
A^{-1}_{2,q^{-2i+2j+2}}
{\raisebox{-4pt}
 {
{\setlength{\unitlength}{0.2mm}
\begin{picture}(25,25)
\multiput(0,0)(25,0){2}{\line(0,1){25}}
\multiput(0,0)(0,25){2}{\line(1,0){25}}
\put(8,7){$2$}
\put(28,-3){$_{ij}$}
\end{picture}
}
}
}
=
{\raisebox{-4pt}
 {
{\setlength{\unitlength}{0.2mm}
\begin{picture}(25,25)
\multiput(0,0)(25,0){2}{\line(0,1){25}}
\multiput(0,0)(0,25){2}{\line(1,0){25}}
\put(8,7){$3$}
\put(28,-3){$_{ij}$}
\end{picture}
}
}
}.
\end{align}
Namely, the multiplication of $A^{-1}_{i,a}$
is regarded as the `action' of changing
the letter $i$ to $i+1$
in a tableau, if $a$ is appropriately chosen.

With this notation, one can concisely keep track and express
the whole process
of the algorithm presented above  by the
{\it semistandard  tableaux}  of shape $(2,1)$
as in Figure \ref{fig:A2}.
Moreover, as a corollary of Figure \ref{fig:A2},
we obtain the {\it tableaux expression\/} of the $q$-character
\begin{align}
\label{eq:A2tableauxsum}
\chi_q(V(m_+)) = \sum_{T\in \mathrm{SST}(2,1)} m(T),
\end{align}
where $\mathrm{SST}(2,1)$ is the set of all the semistandard
tableaux  of shape $(2,1)$.

\begin{rem}
For $\fg$ of classical type,
similar tableaux expressions  to (\ref{eq:A2tableauxsum}) have been
conjectured and partially
proved for a large class of irreducible representations $V(\lambda/\mu)$
parametrized by skew Young diagrams $\lambda/\mu$
\cite{BR,KOS,KS,FR, FM2,NN1,NN2,NN3,H5}.
More precisely,
there is a tableaux expression for the `Jacobi-Trudi-type
determinant' $\chi(\lambda/\mu)$, which lies in the image of the homomorphism $\chi_q$.
For types $A_n$ and $B_n$,
it is known that
$\chi(\lambda/\mu)=\chi_q(V(\lambda/\mu))$ for any skew Young diagram \cite{H5,H7}.
For type $A_n$, this was also shown
in the context of the character of
Yangian $Y(\mathfrak{gl}_n)$ \cite{NT}.
For types $C_n$ and $D_n$  \cite{NN1,NN2,NN3},
for a (non-skew) Young diagram $\lambda$,
it was conjectured that
$\chi(\lambda)=\chi_q(V(\lambda))$.
In general,
$\chi(\lambda/\mu)$ is conjectured to be
the $q$-character of, not $V(\lambda/\mu)$ itself,
but some representation which has $V(\lambda/\mu)$
as a subquotient.
Using this opportunity, let us withdraw our false claim
for types $C_n$ and $D_n$  in \cite{NN1,NN2,NN3}
that
{\it we expect that $\chi(\lambda/\mu)=\chi_q(V(\lambda/\mu))$,
 if $\lambda/\mu$ is connected}.
A counterexample is given by $\fg$ of type $C_2$ with
$\lambda=(2,2,1), \mu=(1)$.
\end{rem}

\begin{rem}
The underlying $U_q(\fg)$-character of $V(m_+)$
is symmetric  under the Dynkin diagram automorphism
1 $\leftrightarrow$ 2.
However, we see in Figure \ref{fig:A2}
that the $U_q(\hg)$-structure of $V(m_+)$
is not so.
Of course, this is not a contradiction, because
the highest weight monomial $m_+$ in
(\ref{eq:A2highest}) are not symmetric
under the automorphism.
\end{rem}

\subsection{Example 2}
To convince the reader further that the FM algorithm is
well designed,
let us give another, a little more complicated example,
where the coefficients of some monomials
in the $q$-character are greater than one.
We consider the case where $\fg$ is of type $C_2$
and the representation $V(m_+)$
has the highest weight monomial
\begin{align}
\label{eq:C2}
m_+ = Y_{2,q^{-1}}Y_{2,q}.
\end{align}
The $U_q(\fg)$-weight of $m_+$ is $\lambda_+ = 2\omega_2$.
We faithfully follow the convention in \cite{FR,FM1};
in particular, $\alpha_2$ is the long root.

Since any monomial occurring in $\chi_q(V(m_+))$ for (\ref{eq:C2}) should occur
in the product $\chi_q(V(Y_{2,q^{-1}}))\chi_q(V(Y_{2,q}))$, and
\begin{align}
\chi_q(V(Y_{2,q^{-1}}))&=
Y_{2,q^{-1}}+ Y_{1,1}Y_{1,q^2}Y^{-1}_{2,q^3}
+ Y_{1,1}Y^{-1}_{1,q^4}\\
&\qquad+ Y^{-1}_{1,q^2}Y^{-1}_{1,q^4}Y_{2,q}
+Y^{-1}_{2,q^5},\notag\\
\chi_q(V(Y_{2,q}))&=
Y_{2,q}+ Y_{1,q^2}Y_{1,q^4}Y^{-1}_{2,q^5}
+ Y_{1,q^2}Y^{-1}_{1,q^6}\\
&\qquad+ Y^{-1}_{1,q^4}Y^{-1}_{1,q^6}Y_{2,q^3}
+Y^{-1}_{2,q^7},\notag
\end{align}
one can immediately see that $m_+$ is the only possible dominant monomial
in $\chi_q(V(m_+))$.
Thus,  $V(m_+)$ is special, and the FM algorithm is applicable.
It also implies that $V(Y_{2,q^{-1}})\otimes V(Y_{2,q})$
is irreducible and isomorphic to $V(m_+)$.
In particular, as a $U_q(\fg)$-representation, $V(m_+)$ is decomposed as
$V_{\omega_2}\otimes V_{\omega_2}
\simeq V_{2\omega_2} \oplus V_{2\omega_1} \oplus V_{0}$
with dimension $5 \times 5 = 14+10+1 $,
and its $U_q(\fg)$-weight diagram is given in
Figure \ref{fig:weightdiagram}.

\begin{figure}
\begin{picture}(100,100)
\put(55,50){ \vector(1,0){15}}
\put(45,55){ \vector(-1,1){15}}
\put(60,55){$\alpha_1$}
\put(25,60){$\alpha_2$}
\put(60,96){$2\omega_2$}
\put(51,96){$1$}
\put(26,71){$2$}
\put(51,71){$2$}
\put(76,71){$2$}
\put(1,46){$1$}
\put(26,46){$2$}
\put(51,46){$5$}
\put(76,46){$2$}
\put(101,46){$1$}
\put(26,21){$2$}
\put(51,21){$2$}
\put(76,21){$2$}
\put(51,-4){$1$}
\end{picture}
\caption{The $U_q(\fg)$-weight diagram of $V(m_+)$
in Example 2.
The numbers represent the weight multiplicities.}
\label{fig:weightdiagram}
\end{figure}

Keep in mind that (Definition \ref{defn:i-expansion} (ii))
the $i$-expansion should be  done, {\it not with $U_q(\hsl_2)$,
 but with $U_{q_i}(\hsl_2)$}.
Then, the algorithm can be straightforwardly executed with the data:
$q_1=q$, $q_2=q^2$, and
\begin{align}
\label{eq:A2}
A^{-1}_{1,a}= Y^{-1}_{1,aq^{-1}}Y^{-1}_{1,aq}Y_{2,a},
\quad
A^{-1}_{2,a}= Y^{-1}_{2,aq^{-2}}Y^{-1}_{2,aq^2}
Y_{1,aq^{-1}}Y_{1,aq}.
\end{align}

Again, the flow of the algorithm can be expressed with
 tableaux of shape $(2,2)$.
We assign a monomial
to each letter $a=1,2,\overline{2},\overline{1}$ within the box
at position $(i,j)$ as (cf.\ \cite[Section 5.4.3]{FR}).
\begin{align}
\begin{split}
\raisebox{-4pt}
 {
{\setlength{\unitlength}{0.2mm}
\begin{picture}(25,25)
\multiput(0,0)(25,0){2}{\line(0,1){25}}
\multiput(0,0)(0,25){2}{\line(1,0){25}}
\put(8,7){$1$}
\put(28,-3){$_{ij}$}
\end{picture}
}
}
&=
Y_{1,q^{-2i+2j}},
\quad
\raisebox{-4pt}
 {
{\setlength{\unitlength}{0.2mm}
\begin{picture}(25,25)
\multiput(0,0)(25,0){2}{\line(0,1){25}}
\multiput(0,0)(0,25){2}{\line(1,0){25}}
\put(8,7){$\overline{2}$}
\put(28,-3){$_{ij}$}
\end{picture}
}
}
=
Y_{1,q^{-2i+2j + 4}}Y^{-1}_{2,q^{-2i+2j+5}},
\\
\raisebox{-4pt}
 {
{\setlength{\unitlength}{0.2mm}
\begin{picture}(25,25)
\multiput(0,0)(25,0){2}{\line(0,1){25}}
\multiput(0,0)(0,25){2}{\line(1,0){25}}
\put(8,7){$2$}
\put(28,-3){$_{ij}$}
\end{picture}
}
}
&=
Y^{-1}_{1,q^{-2i+2j+2}}Y_{2,q^{-2i+2j+1}},
\quad
\raisebox{-4pt}
{
{\setlength{\unitlength}{0.2mm}
\begin{picture}(25,25)
\multiput(0,0)(25,0){2}{\line(0,1){25}}
\multiput(0,0)(0,25){2}{\line(1,0){25}}
\put(8,7){$\overline{1}$}
\put(28,-3){$_{ij}$}
\end{picture}
}
}
=
Y^{-1}_{1,q^{-2i+2j+6}}.
\end{split}
\end{align}
For example, 
the highest weight monomial $m_+$ is
represented as
\begin{align}
m\left(
\raisebox{-12pt}
 {
{\setlength{\unitlength}{0.2mm}
\begin{picture}(50,50)
\multiput(0,0)(25,0){3}{\line(0,1){50}}
\multiput(0,0)(0,25){3}{\line(1,0){50}}
\put(8,32){$1$}
\put(33,32){$1$}
\put(8,7){$2$}
\put(33,7){$2$}
\end{picture}
}
}
\right)
=
Y_{1,1} Y_{1,q^2} (Y^{-1}_{1,1}Y_{2,q^{-1}})(Y^{-1}_{1,q^2}Y_{2,q})
= Y_{2,q^{-1}}Y_{2,q}.
\end{align}
The `action' of $A^{-1}_{i,a}$ on a box is given by
\begin{align}
\label{eq:aaction2}
{\raisebox{-4pt}
 {
{\setlength{\unitlength}{0.2mm}
\begin{picture}(25,25)
\multiput(0,0)(25,0){2}{\line(0,1){25}}
\multiput(0,0)(0,25){2}{\line(1,0){25}}
\put(8,7){$1$}
\put(28,-3){$_{ij}$}
\end{picture}
}
}
}
\overset{
A^{-1}_{1,q^{-2i+2j+1}}
}{\longrightarrow}
{\raisebox{-4pt}
 {
{\setlength{\unitlength}{0.2mm}
\begin{picture}(25,25)
\multiput(0,0)(25,0){2}{\line(0,1){25}}
\multiput(0,0)(0,25){2}{\line(1,0){25}}
\put(8,7){$2$}
\put(28,-3){$_{ij}$}
\end{picture}
}
}
}
\overset{
A^{-1}_{2,q^{-2i+2j+3}}
}{\longrightarrow}
{\raisebox{-4pt}
 {
{\setlength{\unitlength}{0.2mm}
\begin{picture}(25,25)
\multiput(0,0)(25,0){2}{\line(0,1){25}}
\multiput(0,0)(0,25){2}{\line(1,0){25}}
\put(8,7){$\overline{2}$}
\put(28,-3){$_{ij}$}
\end{picture}
}
}
}
\overset{
A^{-1}_{1,q^{-2i+2j+5}}
}{\longrightarrow}
{\raisebox{-4pt}
 {
{\setlength{\unitlength}{0.2mm}
\begin{picture}(25,25)
\multiput(0,0)(25,0){2}{\line(0,1){25}}
\multiput(0,0)(0,25){2}{\line(1,0){25}}
\put(8,7){$\overline{1}$}
\put(28,-3){$_{ij}$}
\end{picture}
}
}
}.
\end{align}
Then, the flow and the
result of the algorithm is expressed
by tableaux in Figure \ref{fig:C2}.
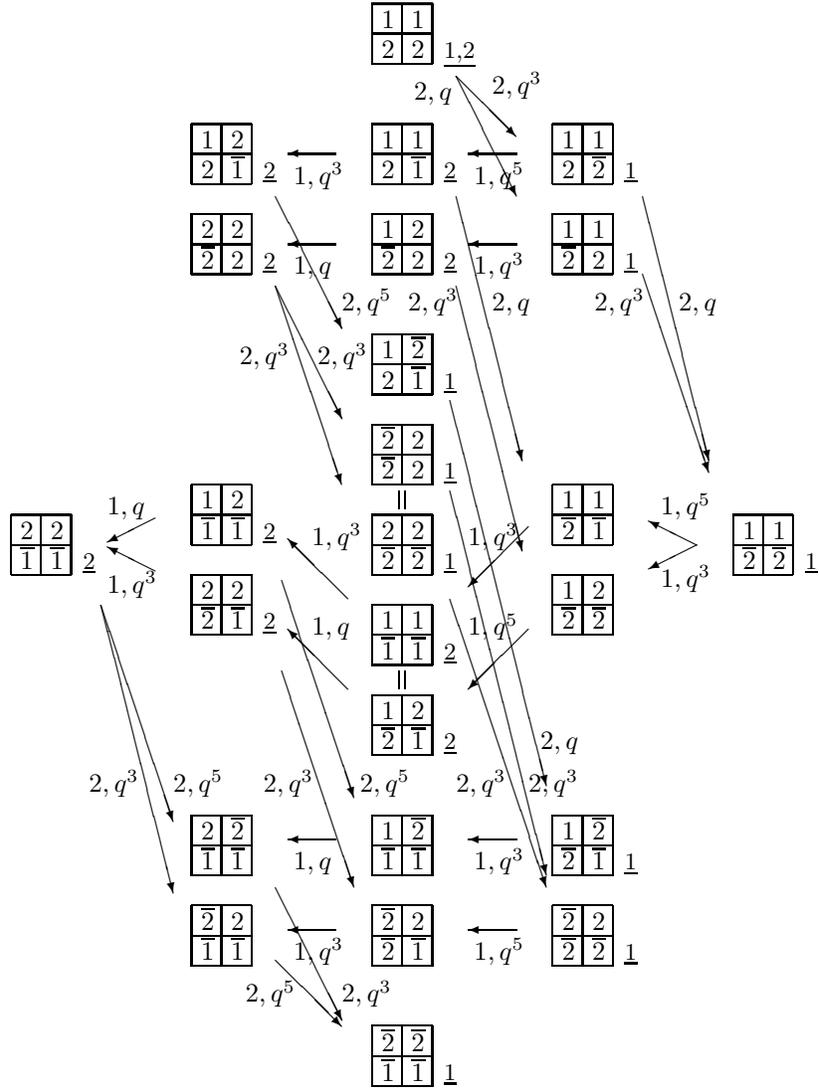
\begin{figure}
\small
\setlength{\unitlength}{0.16mm}
\begin{picture}(650,900)

%
%
\multiput(300,850)(25,0){3}{\line(0,1){50}}
\multiput(300,850)(0,25){3}{\line(1,0){50}}
\put(308,880){$1$}
\put(333,880){$1$}
\put(308,855){$2$}
\put(333,855){$2$}
%
%
\multiput(150,750)(25,0){3}{\line(0,1){50}}
\multiput(150,750)(0,25){3}{\line(1,0){50}}
\put(158,780){$1$}
\put(183,780){$2$}
\put(158,755){$2$}
\put(183,755){$\ol{1}$}
\multiput(150,675)(25,0){3}{\line(0,1){50}}
\multiput(150,675)(0,25){3}{\line(1,0){50}}
\put(158,705){$2$}
\put(183,705){$2$}
\put(158,680){$\ol{2}$}
\put(183,680){$2$}
\multiput(300,750)(25,0){3}{\line(0,1){50}}
\multiput(300,750)(0,25){3}{\line(1,0){50}}
\put(308,780){$1$}
\put(333,780){$1$}
\put(308,755){$2$}
\put(333,755){$\ol{1}$}
\multiput(300,675)(25,0){3}{\line(0,1){50}}
\multiput(300,675)(0,25){3}{\line(1,0){50}}
\put(308,705){$1$}
\put(333,705){$2$}
\put(308,680){$\ol{2}$}
\put(333,680){$2$}
\multiput(450,750)(25,0){3}{\line(0,1){50}}
\multiput(450,750)(0,25){3}{\line(1,0){50}}
\put(458,780){$1$}
\put(483,780){$1$}
\put(458,755){$2$}
\put(483,755){$\ol{2}$}
\multiput(450,675)(25,0){3}{\line(0,1){50}}
\multiput(450,675)(0,25){3}{\line(1,0){50}}
\put(458,705){$1$}
\put(483,705){$1$}
\put(458,680){$\ol{2}$}
\put(483,680){$2$}
%
%
%
\multiput(0,425)(25,0){3}{\line(0,1){50}}
\multiput(0,425)(0,25){3}{\line(1,0){50}}
\put(8,455){$2$}
\put(33,455){$2$}
\put(08,430){$\ol{1}$}
\put(33,430){$\ol{1}$}
%
%
\multiput(150,450)(25,0){3}{\line(0,1){50}}
\multiput(150,450)(0,25){3}{\line(1,0){50}}
\put(158,480){$1$}
\put(183,480){$2$}
\put(158,455){$\ol{1}$}
\put(183,455){$\ol{1}$}
\multiput(150,375)(25,0){3}{\line(0,1){50}}
\multiput(150,375)(0,25){3}{\line(1,0){50}}
\put(158,405){$2$}
\put(183,405){$2$}
\put(158,380){$\ol{2}$}
\put(183,380){$\ol{1}$}
%
\multiput(300,575)(25,0){3}{\line(0,1){50}}
\multiput(300,575)(0,25){3}{\line(1,0){50}}
\put(308,605){$1$}
\put(333,605){$\ol{2}$}
\put(308,580){$2$}
\put(333,580){$\ol{1}$}
\multiput(300,500)(25,0){3}{\line(0,1){50}}
\multiput(300,500)(0,25){3}{\line(1,0){50}}
\put(308,530){$\ol{2}$}
\put(333,530){$2$}
\put(308,505){$\ol{2}$}
\put(333,505){$2$}
\multiput(300,425)(25,0){3}{\line(0,1){50}}
\multiput(300,425)(0,25){3}{\line(1,0){50}}
\put(308,455){$2$}
\put(333,455){$2$}
\put(308,430){$\ol{2}$}
\put(333,430){$\ol{2}$}
\multiput(300,350)(25,0){3}{\line(0,1){50}}
\multiput(300,350)(0,25){3}{\line(1,0){50}}
\put(308,380){$1$}
\put(333,380){$1$}
\put(308,355){$\ol{1}$}
\put(333,355){$\ol{1}$}
\multiput(300,275)(25,0){3}{\line(0,1){50}}
\multiput(300,275)(0,25){3}{\line(1,0){50}}
\put(308,305){$1$}
\put(333,305){$2$}
\put(308,280){$\ol{2}$}
\put(333,280){$\ol{1}$}
%
%
\multiput(450,450)(25,0){3}{\line(0,1){50}}
\multiput(450,450)(0,25){3}{\line(1,0){50}}
\put(458,480){$1$}
\put(483,480){$1$}
\put(458,455){$\ol{2}$}
\put(483,455){$\ol{1}$}
\multiput(450,375)(25,0){3}{\line(0,1){50}}
\multiput(450,375)(0,25){3}{\line(1,0){50}}
\put(458,405){$1$}
\put(483,405){$2$}
\put(458,380){$\ol{2}$}
\put(483,380){$\ol{2}$}
%
%
\multiput(600,425)(25,0){3}{\line(0,1){50}}
\multiput(600,425)(0,25){3}{\line(1,0){50}}
\put(608,455){$1$}
\put(633,455){$1$}
\put(608,430){$\ol{2}$}
\put(633,430){$\ol{2}$}
%
%
\multiput(150,175)(25,0){3}{\line(0,1){50}}
\multiput(150,175)(0,25){3}{\line(1,0){50}}
\put(158,205){$2$}
\put(183,205){$\ol{2}$}
\put(158,180){$\ol{1}$}
\put(183,180){$\ol{1}$}
\multiput(150,100)(25,0){3}{\line(0,1){50}}
\multiput(150,100)(0,25){3}{\line(1,0){50}}
\put(158,130){$\ol{2}$}
\put(183,130){$2$}
\put(158,105){$\ol{1}$}
\put(183,105){$\ol{1}$}
\multiput(300,175)(25,0){3}{\line(0,1){50}}
\multiput(300,175)(0,25){3}{\line(1,0){50}}
\put(308,205){$1$}
\put(333,205){$\ol{2}$}
\put(308,180){$\ol{1}$}
\put(333,180){$\ol{1}$}
\multiput(300,100)(25,0){3}{\line(0,1){50}}
\multiput(300,100)(0,25){3}{\line(1,0){50}}
\put(308,130){$\ol{2}$}
\put(333,130){$2$}
\put(308,105){$\ol{2}$}
\put(333,105){$\ol{1}$}
\multiput(450,175)(25,0){3}{\line(0,1){50}}
\multiput(450,175)(0,25){3}{\line(1,0){50}}
\put(458,205){$1$}
\put(483,205){$\ol{2}$}
\put(458,180){$\ol{2}$}
\put(483,180){$\ol{1}$}
\multiput(450,100)(25,0){3}{\line(0,1){50}}
\multiput(450,100)(0,25){3}{\line(1,0){50}}
\put(458,130){$\ol{2}$}
\put(483,130){$2$}
\put(458,105){$\ol{2}$}
\put(483,105){$\ol{2}$}
%
%
%
\multiput(300,0)(25,0){3}{\line(0,1){50}}
\multiput(300,0)(0,25){3}{\line(1,0){50}}
\put(308,29){$\ol{2}$}
\put(333,29){$\ol{2}$}
\put(308,4){$\ol{1}$}
\put(333,4){$\ol{1}$}
\put(360,855){\footnotesize\underline{1,2}}
\put(210,755){\footnotesize\underline{2}}
\put(210,680){\footnotesize\underline{2}}
\put(360,755){\footnotesize\underline{2}}
\put(360,680){\footnotesize\underline{2}}
\put(510,755){\footnotesize\underline{1}}
\put(510,680){\footnotesize\underline{1}}
\put(60,430){\footnotesize\underline{2}}
\put(210,455){\footnotesize\underline{2}}
\put(210,380){\footnotesize\underline{2}}
\put(360,580){\footnotesize\underline{1}}
\put(360,505){\footnotesize\underline{1}}
\put(360,430){\footnotesize\underline{1}}
\put(360,355){\footnotesize\underline{2}}
\put(360,280){\footnotesize\underline{2}}
\put(660,430){\footnotesize\underline{1}}
\put(510,180){\footnotesize\underline{1}}
\put(510,105){\footnotesize\underline{1}}
\put(360,5){\footnotesize\underline{1}}
\put(322.5,330){\line(0,1){15}}
\put(327.5,330){\line(0,1){15}}
\put(322.5,480){\line(0,1){15}}
\put(327.5,480){\line(0,1){15}}
\put(420,775){\vector(-1,0){40}}
\put(385,750){$1,q^5$}
\put(270,775){\vector(-1,0){40}}
\put(235,750){$1,q^3$}
\put(420,700){\vector(-1,0){40}}
\put(385,675){$1,q^3$}
\put(270,700){\vector(-1,0){40}}
\put(235,675){$1,q$}
\put(570,450){\vector(-2,1){40}}
\put(540,475){$1,q^5$}
\put(570,450){\vector(-2,-1){40}}
\put(540,415){$1,q^3$}
\put(120,472){\vector(-2,-1){40}}
\put( 80,475){$1,q$}
\put(120,428){\vector(-2,1){40}}
\put( 80,410){$1,q^3$}
\put(430,465){\vector(-1,-1){50}}
\put(380,450){$1,q^3$}
\put(430,380){\vector(-1,-1){50}}
\put(380,375){$1,q^5$}
\put(280,405){\vector(-1,1){50}}
\put(250,450){$1,q^3$}
\put(280,330){\vector(-1,1){50}}
\put(250,375){$1,q$}
\put(420,205){\vector(-1,0){40}}
\put(385,180){$1,q^3$}
\put(270,205){\vector(-1,0){40}}
\put(235,180){$1,q$}
\put(420,130){\vector(-1,0){40}}
\put(385,105){$1,q^5$}
\put(270,130){\vector(-1,0){40}}
\put(235,105){$1,q^3$}
\put(370,840){\vector(1,-1){50}}
\put(400,825){$2,q^3$}
\put(370,840){\vector(1,-2){50}}
\put(335,820){$2,q$}
\put(525,740){\vector(1,-4){55}}
\put(555,645){$2,q$}
\put(525,675){\vector(1,-3){55}}
\put(485,645){$2,q^3$}
\put(370,740){\vector(1,-4){55}}
\put(400,645){$2,q$}
\put(370,665){\vector(1,-4){55}}
\put(330,645){$2,q^3$}
\put(220,740){\vector(1,-2){55}}
\put(275,645){$2,q^5$}
\put(220,665){\vector(1,-2){55}}
\put(255,600){$2,q^3$}
\put(220,665){\vector(1,-3){55}}
\put(190,600){$2,q^3$}
\put(365,570){\vector(1,-4){80}}
\put(440,280){$2,q$}
\put(365,495){\vector(1,-4){80}}
\put(430,245){$2,q^3$}
\put(365,405){\vector(1,-3){80}}
\put(370,245){$2,q^3$}
\put(225,420){\vector(1,-3){60}}
\put(290,245){$2,q^5$}
\put(225,345){\vector(1,-3){60}}
\put(210,245){$2,q^3$}
\put(75,400){\vector(1,-3){60}}
\put(135,245){$2,q^5$}
\put(75,400){\vector(1,-4){60}}
\put(65,245){$2,q^3$}
\put(220,165){\vector(1,-2){55}}
\put(275,70){$2,q^3$}
\put(220,105){\vector(1,-1){55}}
\put(195,70){$2,q^5$}
\end{picture}
\caption{ The flow of the FM algorithm by Young
tableaux for Example 2.
The equality between two tableaux means that
they represent the same monomial.
}
\label{fig:C2}
\end{figure}

We note that
two monomials 
occur in $\chi_q(V(m_+))$
with coefficient two,
and, in Figure \ref{fig:C2},
each monomial 
is represented by  two different 
tableaux such as
\begin{align}
\label{eq:indenfityTableaux}
m\left(
\raisebox{-12pt}
 {
{\setlength{\unitlength}{0.2mm}
\begin{picture}(50,50)
\multiput(0,0)(25,0){3}{\line(0,1){50}}
\multiput(0,0)(0,25){3}{\line(1,0){50}}
\put(8,32){$\ol{2}$}
\put(33,32){$2$}
\put(8,7){$\ol{2}$}
\put(33,7){$2$}
\end{picture}
}
}
\right)
=
m\left(
\raisebox{-12pt}
 {
{\setlength{\unitlength}{0.2mm}
\begin{picture}(50,50)
\multiput(0,0)(25,0){3}{\line(0,1){50}}
\multiput(0,0)(0,25){3}{\line(1,0){50}}
\put(8,32){$2$}
\put(33,32){$2$}
\put(8,7){$\ol{2}$}
\put(33,7){$\ol{2}$}
\end{picture}
}
}
\right)
,
\
m\left(
\raisebox{-12pt}
 {
{\setlength{\unitlength}{0.2mm}
\begin{picture}(50,50)
\multiput(0,0)(25,0){3}{\line(0,1){50}}
\multiput(0,0)(0,25){3}{\line(1,0){50}}
\put(8,32){$1$}
\put(33,32){$1$}
\put(8,7){$\ol{1}$}
\put(33,7){$\ol{1}$}
\end{picture}
}
}
\right)
=
m\left(
\raisebox{-12pt}
 {
{\setlength{\unitlength}{0.2mm}
\begin{picture}(50,50)
\multiput(0,0)(25,0){3}{\line(0,1){50}}
\multiput(0,0)(0,25){3}{\line(1,0){50}}
\put(8,32){$1$}
\put(33,32){$2$}
\put(8,7){$\ol{2}$}
\put(33,7){$\ol{1}$}
\end{picture}
}
}
\right).
\end{align}
{}Purely from the point of view of the FM algorithm,
this is redundant,
because the FM algorithm does not
distinguish tableaux if they represents the same
monomial.
However, by doing this,
we have the following {\it tableaux expression\/}
of the $q$-character
\begin{align}
\label{eq:tableauxsum}
\chi_q(V(m_+)) = \sum_{T\in \mathrm{Tab}} m(T),
\end{align}
where $\mathrm{Tab}$ is the set of the tableaux
occurring in Figure \ref{fig:C2}.
Remarkably, the formula (\ref{eq:tableauxsum}) exactly
coincides with the tableaux expression 
in \cite{NN1,NN3} based on the Jacobi-Trudi-type determinant,
thereby showing nice compatibility between two approaches.

\begin{rem}
The implementation of the FM algorithm with tableaux
(or, equivalently, with paths) of \cite{NN1,NN3} demonstrated here,
 can be generalized
to the skew diagram representations of type $C_n$  \cite{NN4}.
See also Remark \ref{rem:concl}.
\end{rem}

\section{Counterexample}

Now we are ready to present an example where
the FM algorithm {\it fails\/} in the sense of
Step 2 (i) of Definition \ref{defn:FMalgorithm}.

We consider the case where $\fg$ is of type $C_3$
and the representation $V(m_+)$ has the highest weight
monomial
\begin{align}
\label{eq:mp3}
m_+ = Y_{1,q^{4}}Y_{2,q}Y_{3,q^{-2}}.
\end{align}
According to \cite[Conjecture 2.2, Theorem A.1]{NN1},
it is expected to
be decomposed into
$V_{\omega_1+\omega_2+\omega_3} \oplus V_{2\omega_1 + \omega_2}
 \oplus V_{2\omega_2}\oplus V_{\omega_1+\omega_3}
\oplus V_{2\omega_1} \oplus V_{\omega_2}$
as a $U_q(\fg)$-representation, 
with dimension $512+189+90+70+21+14 = 896$.
The algorithm is executed with the data:
$q_1=q_2=q$, $q_3=q^2$, and
\begin{align}
\begin{split}
\label{eq:C3}
A^{-1}_{1,a}&= Y^{-1}_{1,aq^{-1}}Y^{-1}_{1,aq}Y_{2,a},
\quad
A^{-1}_{2,a}= Y^{-1}_{2,aq^{-1}}Y^{-1}_{2,aq}
Y_{1,a}Y_{3,a},\\
A^{-1}_{3,a}&= Y^{-1}_{3,aq^{-2}}Y^{-1}_{3,aq^2}
Y_{2,aq^{-1}}Y_{2,aq}.
\end{split}
\end{align}

Again, the process of the algorithm can be expressed by
Young tableaux  of shape $(3,2,1)$.
We assign a monomial
to each letter $a=1,2,3,\overline{3},\overline{2},\overline{1}$ within the box
at position $(i,j)$ as 
\begin{align}
\begin{split}
\label{eq:box3}
\raisebox{-4pt}
 {
{\setlength{\unitlength}{0.2mm}
\begin{picture}(25,25)
\multiput(0,0)(25,0){2}{\line(0,1){25}}
\multiput(0,0)(0,25){2}{\line(1,0){25}}
\put(8,7){$1$}
\put(28,-3){$_{ij}$}
\end{picture}
}
}
&=
Y_{1,q^{-2i+2j}},
\hskip52.5pt
\raisebox{-4pt}
 {
{\setlength{\unitlength}{0.2mm}
\begin{picture}(25,25)
\multiput(0,0)(25,0){2}{\line(0,1){25}}
\multiput(0,0)(0,25){2}{\line(1,0){25}}
\put(8,7){$\overline{3}$}
\put(28,-3){$_{ij}$}
\end{picture}
}
}
=
Y_{2,q^{-2i+2j + 5}}Y^{-1}_{3,q^{-2i+2j+6}},
\\
\raisebox{-4pt}
 {
{\setlength{\unitlength}{0.2mm}
\begin{picture}(25,25)
\multiput(0,0)(25,0){2}{\line(0,1){25}}
\multiput(0,0)(0,25){2}{\line(1,0){25}}
\put(8,7){$2$}
\put(28,-3){$_{ij}$}
\end{picture}
}
}
&=
Y^{-1}_{1,q^{-2i+2j+2}}Y_{2,q^{-2i+2j+1}}
,
\hskip-5pt
\raisebox{-4pt}
 {
{\setlength{\unitlength}{0.2mm}
\begin{picture}(25,25)
\multiput(0,0)(25,0){2}{\line(0,1){25}}
\multiput(0,0)(0,25){2}{\line(1,0){25}}
\put(8,7){$\overline{2}$}
\put(28,-3){$_{ij}$}
\end{picture}
}
}
=
Y_{1,q^{-2i+2j + 6}}Y^{-1}_{2,q^{-2i+2j+7}},
\\
\raisebox{-4pt}
 {
{\setlength{\unitlength}{0.2mm}
\begin{picture}(25,25)
\multiput(0,0)(25,0){2}{\line(0,1){25}}
\multiput(0,0)(0,25){2}{\line(1,0){25}}
\put(8,7){$3$}
\put(28,-3){$_{ij}$}
\end{picture}
}
}
&=
Y^{-1}_{2,q^{-2i+2j+3}}Y_{3,q^{-2i+2j+2}},
\hskip-5pt
\raisebox{-4pt}
{
{\setlength{\unitlength}{0.2mm}
\begin{picture}(25,25)
\multiput(0,0)(25,0){2}{\line(0,1){25}}
\multiput(0,0)(0,25){2}{\line(1,0){25}}
\put(8,7){$\overline{1}$}
\put(28,-3){$_{ij}$}
\end{picture}
}
}
=
Y^{-1}_{1,q^{-2i+2j+8}}.
\end{split}
\end{align}
For example, 
the highest weight monomial $m_+$ is represented as
\begin{align}
\begin{split}
&\ m\left(
\raisebox{-15pt}
 {
{\setlength{\unitlength}{0.16mm}
\begin{picture}(75,75)
\multiput(0,0)(25,0){2}{\line(0,1){75}}
\put(50,25){\line(0,1){50}}
\put(75,50){\line(0,1){25}}
\put(0,0){\line(1,0){25}}
\put(0,25){\line(1,0){50}}
\multiput(0,50)(0,25){2}{\line(1,0){75}}
\put(8,56){$1$}
\put(33,56){$1$}
\put(58,56){$1$}
\put(8,31){$2$}
\put(33,31){$2$}
\put(8,6){$3$}
\end{picture}
}
}
\right)\\
=&
\ Y_{1,1} Y_{1,q^2}Y_{1,q^4}
 (Y^{-1}_{1,1}Y_{2,q^{-1}})(Y^{-1}_{1,q^2}Y_{2,q})
(Y^{-1}_{2,q^{-1}}Y_{3,q^{-2}})\\
=&
\ Y_{1,q^{4}}Y_{2,q}Y_{3,q^{-2}}.
\end{split}
\end{align}
The `action' of $A^{-1}_{i,a}$ on a box is given by
\begin{align}
\label{eq:aaction3}
\begin{split}
{\raisebox{-4pt}
 {
{\setlength{\unitlength}{0.2mm}
\begin{picture}(25,25)
\multiput(0,0)(25,0){2}{\line(0,1){25}}
\multiput(0,0)(0,25){2}{\line(1,0){25}}
\put(8,7){$1$}
\put(28,-3){$_{ij}$}
\end{picture}
}
}
}
\overset{
A^{-1}_{1,q^{-2i+2j+1}}
}{\longrightarrow}
{\raisebox{-4pt}
 {
{\setlength{\unitlength}{0.2mm}
\begin{picture}(25,25)
\multiput(0,0)(25,0){2}{\line(0,1){25}}
\multiput(0,0)(0,25){2}{\line(1,0){25}}
\put(8,7){$2$}
\put(28,-3){$_{ij}$}
\end{picture}
}
}
}
\overset{
A^{-1}_{2,q^{-2i+2j+2}}
}{\longrightarrow}
{\raisebox{-4pt}
 {
{\setlength{\unitlength}{0.2mm}
\begin{picture}(25,25)
\multiput(0,0)(25,0){2}{\line(0,1){25}}
\multiput(0,0)(0,25){2}{\line(1,0){25}}
\put(8,7){$3$}
\put(28,-3){$_{ij}$}
\end{picture}
}
}
}
\overset{
A^{-1}_{3,q^{-2i+2j+4}}
}{\longrightarrow}\\
{\raisebox{-4pt}
 {
{\setlength{\unitlength}{0.2mm}
\begin{picture}(25,25)
\multiput(0,0)(25,0){2}{\line(0,1){25}}
\multiput(0,0)(0,25){2}{\line(1,0){25}}
\put(8,7){$\overline{3}$}
\put(28,-3){$_{ij}$}
\end{picture}
}
}
}
\overset{
A^{-1}_{2,q^{-2i+2j+6}}
}{\longrightarrow}
{\raisebox{-4pt}
 {
{\setlength{\unitlength}{0.2mm}
\begin{picture}(25,25)
\multiput(0,0)(25,0){2}{\line(0,1){25}}
\multiput(0,0)(0,25){2}{\line(1,0){25}}
\put(8,7){$\overline{2}$}
\put(28,-3){$_{ij}$}
\end{picture}
}
}
}
\overset{
A^{-1}_{1,q^{-2i+2j+7}}
}{\longrightarrow}
{\raisebox{-4pt}
 {
{\setlength{\unitlength}{0.2mm}
\begin{picture}(25,25)
\multiput(0,0)(25,0){2}{\line(0,1){25}}
\multiput(0,0)(0,25){2}{\line(1,0){25}}
\put(8,7){$\overline{1}$}
\put(28,-3){$_{ij}$}
\end{picture}
}
}
}.
\end{split}
\end{align}

\begin{thm}
\label{thm:C3}
The FM algorithm fails for $m_+$ in (\ref{eq:mp3}).
\end{thm}

Let us prove the theorem.
We set
\begin{align}
\label{eq:m1}
m_1 & := A^{-1}_{3,1} m_+
= Y_{1,q^4}(Y_{2,q^{-1}}Y^2_{2,q})Y^{-1}_{3,q^2}
,\\
m_2 & := 
 A^{-1}_{2,q^2}A^{-1}_{3,1} m_+
=(Y_{1,q^2}Y_{1,q^4})(Y_{2,q^{-1}}Y_{2,q}Y^{-1}_{2,q^3})
,\\
\label{eq:m3}
m_3 & := A^{-2}_{2,q^2}A^{-1}_{3,1} m_+
=(Y^2_{1,q^2}Y_{1,q^4})(Y_{2,q^{-1}}Y^{-2}_{2,q^3})
Y_{3,q^2}
,\\
\label{eq:m4}
m_4 & := A^{-1}_{1,q^3}A^{-2}_{2,q^2}A^{-1}_{3,1} m_+
=Y_{1,q^2}(Y_{2,q^{-1}}Y^{-1}_{2,q^3})
Y_{3,q^2}
,
,\\
\label{eq:m5}
m_5 & := A^{-1}_{1,q^3}A^{-1}_{2,q^2}A^{-1}_{3,1} m_+
=Y_{2,q^{-1}}Y_{2,q}
,\\
m_6 & := A^{-1}_{2,q^2} m_+
=(Y_{1,q^2}Y_{1,q^4})(Y_{3,q^{-2}}Y_{3,q^2})
.
\end{align}
We show below that the algorithm fails at $m_4$.
See Figure \ref{fig:C3} 
for the outline of the proof  in 
terms of tableaux.

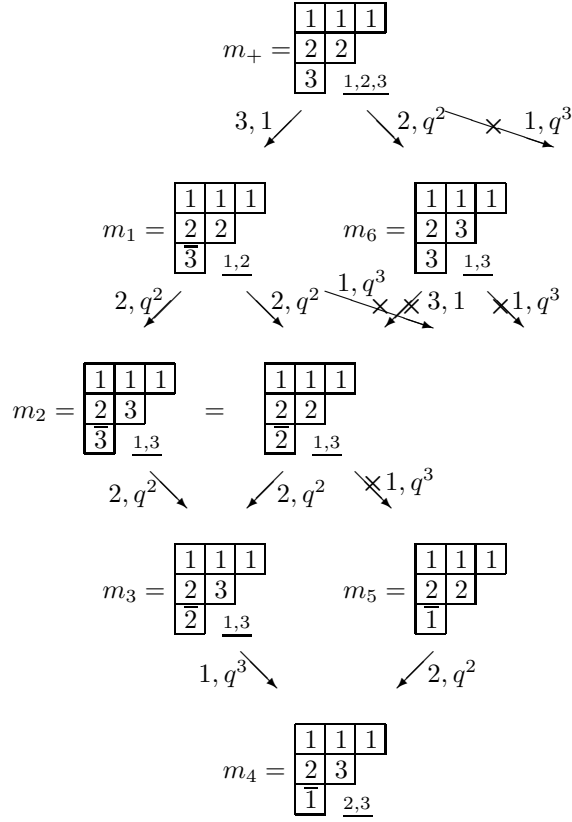
\begin{figure}
\small
\setlength{\unitlength}{0.16mm}
\begin{picture}(450,675)(-10,0)

%
%
\multiput(225,600)(25,0){2}{\line(0,1){75}}
\put(275,625){\line(0,1){50}}
\put(300,650){\line(0,1){25}}
\put(225,600){\line(1,0){25}}
\put(225,625){\line(1,0){50}}
\multiput(225,650)(0,25){2}{\line(1,0){75}}
\put(233,655){$1$}
\put(258,655){$1$}
\put(283,655){$1$}
\put(233,630){$2$}
\put(258,630){$2$}
\put(233,605){$3$}
\put(165,630){$ m_+=$}
%
%
\multiput(125,450)(25,0){2}{\line(0,1){75}}
\put(175,475){\line(0,1){50}}
\put(200,500){\line(0,1){25}}
\put(125,450){\line(1,0){25}}
\put(125,475){\line(1,0){50}}
\multiput(125,500)(0,25){2}{\line(1,0){75}}
\put(133,505){$1$}
\put(158,505){$1$}
\put(183,505){$1$}
\put(133,480){$2$}
\put(158,480){$2$}
\put(133,455){$\ol{3}$}
\put(65,480){$ m_1=$}
\multiput(325,450)(25,0){2}{\line(0,1){75}}
\put(375,475){\line(0,1){50}}
\put(400,500){\line(0,1){25}}
\put(325,450){\line(1,0){25}}
\put(325,475){\line(1,0){50}}
\multiput(325,500)(0,25){2}{\line(1,0){75}}
\put(333,505){$1$}
\put(358,505){$1$}
\put(383,505){$1$}
\put(333,480){$2$}
\put(358,480){$3$}
\put(333,455){$3$}
\put(265,480){$ m_6=$}
%
%
\multiput(50,300)(25,0){2}{\line(0,1){75}}
\put(100,325){\line(0,1){50}}
\put(125,350){\line(0,1){25}}
\put(50,300){\line(1,0){25}}
\put(50,325){\line(1,0){50}}
\multiput(50,350)(0,25){2}{\line(1,0){75}}
\put(58,355){$1$}
\put(83,355){$1$}
\put(108,355){$1$}
\put(58,330){$2$}
\put(83,330){$3$}
\put(58,305){$\ol{3}$}
\put(-10,330){$ m_2=$}
\multiput(200,300)(25,0){2}{\line(0,1){75}}
\put(250,325){\line(0,1){50}}
\put(275,350){\line(0,1){25}}
\put(200,300){\line(1,0){25}}
\put(200,325){\line(1,0){50}}
\multiput(200,350)(0,25){2}{\line(1,0){75}}
\put(208,355){$1$}
\put(233,355){$1$}
\put(258,355){$1$}
\put(208,330){$2$}
\put(233,330){$2$}
\put(208,305){$\ol{2}$}
\put(150,330){$ =$}
%
%
\multiput(125,150)(25,0){2}{\line(0,1){75}}
\put(175,175){\line(0,1){50}}
\put(200,200){\line(0,1){25}}
\put(125,150){\line(1,0){25}}
\put(125,175){\line(1,0){50}}
\multiput(125,200)(0,25){2}{\line(1,0){75}}
\put(133,205){$1$}
\put(158,205){$1$}
\put(183,205){$1$}
\put(133,180){$2$}
\put(158,180){$3$}
\put(133,155){$\ol{2}$}
\put(65,180){$ m_3=$}
\multiput(325,150)(25,0){2}{\line(0,1){75}}
\put(375,175){\line(0,1){50}}
\put(400,200){\line(0,1){25}}
\put(325,150){\line(1,0){25}}
\put(325,175){\line(1,0){50}}
\multiput(325,200)(0,25){2}{\line(1,0){75}}
\put(333,205){$1$}
\put(358,205){$1$}
\put(383,205){$1$}
\put(333,180){$2$}
\put(358,180){$2$}
\put(333,155){$\ol{1}$}
\put(265,180){$ m_5=$}
%
%
\multiput(225,0)(25,0){2}{\line(0,1){75}}
\put(275,25){\line(0,1){50}}
\put(300,50){\line(0,1){25}}
\put(225,0){\line(1,0){25}}
\put(225,25){\line(1,0){50}}
\multiput(225,50)(0,25){2}{\line(1,0){75}}
\put(233,55){$1$}
\put(258,55){$1$}
\put(283,55){$1$}
\put(233,30){$2$}
\put(258,30){$3$}
\put(233,5){$\ol{1}$}
\put(165,30){$ m_4=$}
\put(180,135){\vector(1,-1){30}}
\put(340,135){\vector(-1,-1){30}}
\put(145,110){$1,q^3$}
\put(335,110){$2,q^2$}
\put(105,285){\vector(1,-1){30}}
\put(215,285){\vector(-1,-1){30}}
\put(275,285){\vector(1,-1){30}}
\put(70,260){$2,q^2$}
\put(210,260){$2,q^2$}
\put(300,270){$1,q^3$}
\put(130,435){\vector(-1,-1){30}}
\put(185,435){\vector(1,-1){30}}
\put(330,435){\vector(-1,-1){30}}
\put(385,435){\vector(1,-1){30}}
\put(250,435){\vector(3,-1){90}}
\put(75,420){$2,q^2$}
\put(205,420){$2,q^2$}
\put(405,420){$1,q^3$}
\put(260,435){$1,q^3$}
\put(335,420){$3,1$}
\put(230,585){\vector(-1,-1){30}}
\put(285,585){\vector(1,-1){30}}
\put(350,585){\vector(3,-1){90}}
\put(175,570){$3,1$}
\put(310,570){$2,q^2$}
\put(415,570){$1,q^3$}
\put(280,270){$\boldsymbol\times$}
\put(286,417){$\boldsymbol\times$}
\put(312,417){$\boldsymbol\times$}
\put(386,417){$\boldsymbol\times$}
\put(380,567){$\boldsymbol\times$}
\put(265,608){$\underline{_{1,2,3}}$}
\put(165,458){$\underline{_{1,2}}$}
\put(365,458){$\underline{_{1,3}}$}
\put(90,308){$\underline{_{1,3}}$}
\put(240,308){$\underline{_{1,3}}$}
\put(165,158){$\underline{_{1,3}}$}
\put(265,8){$\underline{_{2,3}}$}

\end{picture}
\caption{The diagram explaining
how the FM algorithm fails for $m_+$
in (\ref{eq:mp3}).}
\label{fig:C3}
\end{figure}

\begin{lem}
The monomial $m_4$ occurs in $\chi$ at some step
in the algorithm.
\end{lem}

\begin{proof}
By (\ref{eq:mp3}),
the $3$-expansion of $\chi$ with respect to $m_+$ gives
$
\mu = m_+ ( 1 + A^{-1}_{3,1}),
$
where $\mu$ is the polynomial in (\ref{eq:sl2ch}).
Therefore, $m_1$ occurs in $\chi$ after the expansion.
Next, by (\ref{eq:m1}), $m_1$ is admissible,
and the $2$-expansion of $\chi$ with respect to $m_1$ gives
$
\mu = m_1 ( 1 + A^{-1}_{2,q^2} + A^{-1}_{2,1}A^{-1}_{2,q^2})
(1+A^{-1}_{2,q^2}).
$
Thus, $m_3$ occurs in $\chi$ after the expansion.
Finally, 
by (\ref{eq:m3}), $m_3$ is admissible,
and the $1$-expansion of $\chi$ with respect to $m_3$ gives
$
\mu = m_1 ( 1 + A^{-1}_{1,q^5} + A^{-1}_{1,q^3}A^{-1}_{1,q^5})
(1+A^{-1}_{1,q^3}).
$
In particular, $m_4$ occurs in $\chi$ after the expansion.
\end{proof}

Let $\lambda\, (=\omega_1+\omega_3)$ denote 
the $U_q(\fg)$-weight of $m_4$. 
Let us show that
the monomial $m_4$ is {\it not\/} admissible when 
$\chi$ is going to be expanded  at $\lambda$;
hence, the algorithm fails at $m_4$.
To see it, suppose that $m_4$ is admissible
 when 
$\chi$ is going to be expanded $\chi$  at $\lambda$.
Since $m_4$ is not $2$-dominant,
it should occur in the $2$-expansion
with respect to
 some $2$-dominant monomial, say, $n$   whose $U_q(\fg)$-weight is
greater than $\lambda$.
Since $\{ A_{i,a}\}_{i\in I;a\in \bbC^\times}$
are algebraically independent, 
$n$ should be either
$m_5 = A_{2,q^2}m_4$ or $m'_5 = A^2_{2,q^2}m_4$.
Then, one can easily check that
the $2$-expansion with respect to $m_5$
generates $m_4$,
while the $2$-expansion with respect to $m'_5$
does not so.
Therefore, $n=m_5$. However,

\begin{lem}
The monomial $m_5$ does not occur in $\chi$ at any step
in the algorithm.
\end{lem}
\begin{proof}
By (\ref{eq:m5}),
there are six possible routes to obtain
$m_5$ from $m_+$ by 
 $i$-expansions:
(The  symbol $\overset{
i,q^{k}
}{\longrightarrow}$ represents the action of $A^{-1}_{i,q^k}$.)

\par
(i)  $m_+ \overset{
1,q^{3}
}{\longrightarrow}
\ast
\overset{
2,q^2
}{\longrightarrow}
\ast
\overset{
3,1
}{\longrightarrow}
m_5$.
The 1-expansion of $\chi$ with respect to $m_+$
gives $\mu=m_+(1 + A^{-1}_{1,q^5})$.
So, it does not happen.

\par
(ii)  $m_+ \overset{
1,q^{3}
}{\longrightarrow}
\ast
\overset{
3,1
}{\longrightarrow}
\ast
\overset{
2,q^2
}{\longrightarrow}
m_5$.
By the same reason as above, it does not happen.

(iii) $m_+ \overset{
2,q^{2}
}{\longrightarrow}
m_6
\overset{
1,q^{3}
}{\longrightarrow}
\ast
\overset{
3,1
}{\longrightarrow}
m_5$.
The $2$-expansion of $\chi$ with respect to $m_+$
gives $\mu=m_+(1 + A^{-1}_{2,q^2})$.
So, $m_6$ occurs in $\chi$.
Then, the 1-expansion of $\chi$ with respect to $m_6$
gives $\mu=m_6(1 + A^{-1}_{1,q^5}+ A^{-1}_{1,q^3}
A^{-1}_{1,q^5})$. So, it does not happen.

(iv) $m_+ \overset{
2,q^{2}
}{\longrightarrow}
m_6
\overset{
3,1
}{\longrightarrow}
m_2
\overset{
1,q^3
}{\longrightarrow}
m_5$.
The $3$-expansion of $\chi$ with respect to $m_6$
gives $\mu=m_6(1 + A^{-1}_{3,q^4}+ A^{-1}_{3,1}
A^{-1}_{3,q^4})$. So, it does not happen.

(v) $m_+ \overset{
3,1
}{\longrightarrow}
m_1
\overset{
1,q^{3}
}{\longrightarrow}
\ast
\overset{
2,q^2
}{\longrightarrow}
m_5$.
The $1$-expansion of $\chi$ with respect to $m_1$
gives $\mu=m_1(1 + A^{-1}_{1,q^5})$. So, it does not happen.

(vi) $m_+ \overset{
3,1
}{\longrightarrow}
m_1
\overset{
2,q^{2}
}{\longrightarrow}
m_2
\overset{
1,q^3
}{\longrightarrow}
m_5$.
The $1$-expansion of $\chi$ with respect to $m_2$
gives $\mu=m_2(1 + A^{-1}_{1,q^5}+ A^{-1}_{1,q^3}
A^{-1}_{1,q^5})$. So, it does not happen.

Therefore, $m_5$ does not occur in $\chi$ at any step.
\end{proof}

This completes the proof of Theorem \ref{thm:C3}.

Shortly speaking, the algorithm fails because it
fails to generate $m_5$ which is
an extra {\it dominant\/} monomial 
in $\chi_q(V(m_+))$.

It is not difficult to find some other examples where
similar phenomena happen.
For example,
it is a good exercise to check that,
if $\fg$ is of type $D_4$ and
the representation has the highest weight monomial
\begin{align}
m\left(
\raisebox{-15pt}
 {
{\setlength{\unitlength}{0.16mm}
\begin{picture}(50,75)
\multiput(0,0)(25,0){2}{\line(0,1){75}}
\put(50,50){\line(0,1){25}}
\put(0,0){\line(1,0){25}}
\put(0,25){\line(1,0){25}}
\multiput(0,50)(0,25){2}{\line(1,0){50}}
\put(8,56){$1$}
\put(33,56){$1$}
\put(8,31){$2$}
\put(8,6){$3$}
\end{picture}
}
}
\right)
=
 Y_{1,q^{2}}Y_{3,q^{-2}}Y_{4,q^{-2}},
\end{align}
the FM algorithm fails  at the monomial
\begin{align}
m\left(
\raisebox{-15pt}
 {
{\setlength{\unitlength}{0.16mm}
\begin{picture}(50,75)
\multiput(0,0)(25,0){2}{\line(0,1){75}}
\put(50,50){\line(0,1){25}}
\put(0,0){\line(1,0){25}}
\put(0,25){\line(1,0){25}}
\multiput(0,50)(0,25){2}{\line(1,0){50}}
\put(8,56){$1$}
\put(33,56){$1$}
\put(8,31){$3$}
\put(8,4){$\ol{1}$}
\end{picture}
}
}
\right)
=
Y_{1,1} Y^{-1}_{2,q}Y_{3,1}Y_{4,1},
\end{align}
where we use the diagrammatic notation
in \cite{FR,NN1,NN2}.

We conclude  with a remark on 
a modification of the FM algorithm.

\begin{rem}
\label{rem:concl}
Actually, in the counterexample above, the FM algorithm
{\it almost} works except for missing one monomial $m_5$.
It suggests the following {\it modification\/} of the algorithm:
when we encounter the non-admissible monomial
$m_4$ in the algorithm, one simply adds $m_5$
 (the `$2$-ancestor' of $m_4$) to $\chi$ with coloring $(0,0,0)$,
then restart the expansions from $\lambda=2\omega_2$.
Then, we have checked by computer that
the modified algorithm stops and certainly generates monomials
represented by 896 tableaux as expected
in \cite{NN1,NN3}.
For general representations, this {\it trace-back}
procedure is, {\it a priori}, not well-defined, because one cannot uniquely
determine the `$i$-ancestor' of a given monomial.
However, for the family of the skew diagram
representations of type $C_n$ in \cite{NN1,NN3},
one can do so {\it with help of tableaux representation (or, more conveniently, paths representation) }
of monomials.
Observe Figure \ref{fig:C3} as a simple example.
By modifying the FM algorithm with the trace-back
procedure, we expect that Conjecture \ref{conj:FMconjecture}
is true for these representations, and it is supported by our computer
experiment.
The detail will be published elsewhere  \cite{NN4}.
\end{rem}

\end{document}